\documentclass{amsart}
\usepackage[utf8]{inputenc}

\def\NZQ{\mathbb}               
\def\NN{{\NZQ N}}
\def\QQ{{\NZQ Q}}

\def\F2{{\NZQ F}_2}

\def\opn#1#2{\def#1{\operatorname{#2}}} 
\opn\chara{char} \opn\length{\ell} \opn\pd{pd} \opn\rk{rk}
\opn\projdim{proj\,dim} \opn\injdim{inj\,dim} \opn\rank{rank}
\opn\depth{depth} \opn\codepth{codepth} \opn\grade{grade}
\opn\height{height} \opn\embdim{emb\,dim} \opn\codim{codim}

\opn\Tr{Tr} \opn\bigrank{big\,rank}
\opn\superheight{superheight}\opn\lcm{lcm}
\opn\trdeg{tr\,deg}%
\opn\reg{reg} \opn\lreg{lreg} \opn\skel{skel}
\opn\Gr{Gr}
\opn\ann{ann}
\opn\sign{sign}
\opn\del{del}
\opn\lex{lex}

\opn\div{div} \opn\Div{Div} \opn\cl{cl} \opn\Cl{Cl}
\opn\Spec{Spec} \opn\Supp{Supp} \opn\supp{supp} \opn\Sing{Sing}
\opn\Ass{Ass}\opn\fdepth{fdepth}
\opn\Ann{Ann} \opn\Rad{Rad} \opn\Soc{Soc}
\opn\Sym{Sym} \opn\Ker{Ker} \opn\Coker{Coker} \opn\Im{Im}
\opn\Hom{Hom} \opn\Tor{Tor} \opn\Ext{Ext} \opn\End{End}
\opn\Aut{Aut} \opn\id{id} \opn\ini{in} \opn\tr{tr}

\opn\nat{nat}\opn\it{it}
\opn\pff{proof}
\opn\Pf{proof} \opn\GL{GL} \opn\SL{SL} \opn\mod{mod} \opn\ord{ord}
\opn\aff{aff} \opn\con{conv} \opn\relint{relint} 
\opn\lk{lk} \opn\cn{cn} \opn\core{core} \opn\vol{vol}
\opn\link{link} \opn\star{star} \opn\skel{skel} \opn\indeg{indeg}
\opn\Ass{Ass} \opn\Min{Min} \opn\sdepth{sdepth} \opn\depth{depth}
\opn\gr{gr}

\def\pot#1#2{#1[\kern-0.28ex[#2]\kern-0.28ex]}

\opn\dirlim{\underrightarrow{\lim}}
\opn\inivlim{\underleftarrow{\lim}}
\let\to=\rightarrow

\def\Implies{\ifmmode\Longrightarrow \else
     \unskip${}\Longrightarrow{}$\ignorespaces\fi}
\def\implies{\ifmmode\Rightarrow \else
     \unskip${}\Rightarrow{}$\ignorespaces\fi}
\def\iff{\ifmmode\Longleftrightarrow \else
     \unskip${}\Longleftrightarrow{}$\ignorespaces\fi}

\let\:=\colon

\let\ol=\overline
\theoremstyle{plain}
\newtheorem{Theorem}{Theorem}[section]
 \newtheorem{Lemma}[Theorem]{Lemma}
 
 \newtheorem{Proposition}[Theorem]{Proposition}

 \theoremstyle{definition}
 \newtheorem{Definition}[Theorem]{Definition}
  \newtheorem{Algorithm}[Theorem]{Algorithm}
 \newtheorem{Remark}[Theorem]{Remark}
 
 \newtheorem{Example}[Theorem]{Example}

\let\epsilon\varepsilon
\let\kappa=\varkappa
\opn\dis{dis}
\def\pnt{{\raise0.5mm\hbox{\large\bf.}}}

\opn\Lex{Lex}

\usepackage{subcaption,tikz}
\usetikzlibrary{patterns}

\usepackage{float}
\usepackage{tikz-cd}
\usepackage{mathtools}
\usepackage{amsfonts}
\usepackage{amssymb}
\newcommand{\PP}{\mathcal{P}}
\newcommand{\FF}{\mathbb{F}}
\newcommand{\MF}{\mathcal{F}}
\newcommand{\MV}{\mathcal{V}}
\newcommand{\KK}{\mathbb{K}}

\renewcommand{\H}{\mathrm{H}}

\newcommand{\HP}{\mathrm{HP}}
\newcommand{\HS}{\mathrm{HS}}
\renewcommand{\reg}{\mathrm{reg}}

\newcommand{\MM}{\mathcal{M}}
\newcommand{\MN}{\mathcal{N}}
\newcommand{\MO}{\mathcal{O}}
\renewcommand{\QQ}{\mathcal{Q}}

\usepackage{enumitem}
\usepackage{soul}

\title{OliVier: an Oil and Vinegar based cryptosystem}
\author{Antonio Corbo Esposito, Rosa Fera, Francesco Romeo}
\date{}

\begin{document}
\begin{abstract}
    In this paper, we present OliVier a new Public Key Exchange cryptosystem that is based on a multivariate quadratic polynomial system: Oil $\&$ Vinegar polynomials together with fully quadratic ones. We describe its designing process, usage, complexity.
\end{abstract}

\maketitle

\section{Introduction}
In the last years there is a lot of interest in quantum computers' development since they could solve many mathematical problems that are difficult for conventional computers. In particular quantum computers are able to break the most widely used public-key cryptosystems; for example, the ECDSA protocol used for authentication in cryptocurrencies schemes is very well-known to be vulnerable to Shor's Algorithm (see \cite{Sh,HC}): such algorithm can be efficiently implemented on a quantum computer. 
These issues led the National Institute of Standards and Technology (NIST) to open a call in 2016 for ``Post-Quantum'' encryption algorithms. The goal to reach is to create new cryptographic protocols that are secure against both quantum and classical computers. 

All mathematical problems that are known to be NP-complete are natural candidates for building algorithms that satisfy such requirements.
One of these problems is the resolution of Systems of Multivariate Polynomials over Finite Fields, that is NP-complete even for systems of quadratic equations on $\FF_2$.
This feature inspired many researchers to build cryptosystems based on Multivariate Quadratic systems (MQ for brevity), i.e. algebraic systems of multivariate polynomials of degree~2.

The first MQ scheme has been introduced by Matsumoto and Imai in 1988 (see \cite{MI}). Although their $C^{\ast}$ scheme was broken by Jacques Patarin in 1995 (see \cite{Pa1}), such a scheme was generalised by Patarin himself: in 1996, in \cite{Pa2} he presented the \textit{Hidden Field Equations} (HFE) scheme that can be used in the context of asymmetric cryptography, for digital signatures or encryption of very short messages. The crucial point is the existence of a central function $F$, called \emph{trapdoor}, that makes it easy to invert. 

However, there have been many attacks that broke HFE scheme: the first attack was presented by Kipnis and Shamir in 1999 (see \cite{KS2}) by an innovative technique called relinearization; other attacks are the XL algorithm (see \cite{CKPS}), based on the linearization of the system or the one due to Faugere and Joux (see \cite{FJ}), which is based on the computation of Gr\"obner basis of the multivariate quadratic system.\\
However in 1997, Patarin proposed a new scheme, only suitable for digital signatures, called \textit{Oil and Vinegar scheme} (OV scheme) \cite{Pa3}. The main feature of these multivariate quadratic systems relies on separating the variables into two distinct subsets: the subset of variables called ``oil'' and the one of the variables called ``vinegar''. Each polynomial (of degree 2 with coefficients in a finite field) is chosen such that, after an evaluation in the vinegar variables, it is linear only in the oil variables. Such a scheme is based on an under-determined MQ system and the goal is to find one of the many preimages of the original message. 

It must be highlighted that in the case in which the number of oil variables is equal to the number of the vinegar variables, the scheme was already broken by Kipnis and Shamir (see \cite{KS1}), therefore it was modified to the so-called ``Unbalanced Oil and Vinegar scheme'' (UOV), where the number of vinegar variables is greater than the number of oil ones (see \cite{KPG}). Among the various algorithms based on the UOV scheme, there is one that has a lot of success and has been selected by the third round of NIST competition, that is Rainbow. It is a digital signature scheme designed by Ding and Schmidt in 2004 (see \cite{DS}).

In our opinion, the initial failure of HFE and the challenges of inverting an OV system with numerous vinegar variables led many researchers to disregard the potential use of the OV scheme for overdetermined systems.

On the other hand, the early success of UOV and Rainbow multiplied interest and research into building attacks on this type of systems; these studies culminated very recently into very important breakthroughs: Bardet et al. introduced a very efficient attack, Support Minors Modeling, to solve the so-called MinRank problem, i. e. how, given a finite set of matrices, to find their existing linear combinations having a prescribed small rank (see \cite{Ba1}). Such improvement, gave an important contribution to more targeted attacks to UOV and Rainbow, whose security was significantly weakened by Beullens (see \cite{Be1}, \cite{Be2}). Ultimately, this led to the defeat of Rainbow in the final round of NIST competition for signature protocols.

The purpose of the present paper is to show that it is possible to build an MQ-system with trapdoor (therefore suitable not only for signature protocols but also for encryption ones), based on a suitable mixing of OV-type equations and fully quadratic ones, such that, up to the current knowledge, there is no known attack running in a time bounded by polynomial function of the decryption time. 
In \cite{CFR}, we have already analyzed Hilbert series of OV homogeneous systems and those of mixed systems ones. Hilbert series is a useful tool to estimate the degree of regularity of the systems. In this paper, we go a step further and improve the above results, finding some monomial subspaces of the Oil $\&$ Vinegar monomials that, when full rank, provide valuable information about the first fall degree of the system on $\FF_2[x_1,\ldots, x_n]$ modulo the field equations. Moreover, we introduce OliVier, an Oil $\&$ Vinegar based cryptosystem, that has a standard structure, that is a central map $F: \FF_2^n \to \FF_2^{n}$, consisting in OV quadratic equations, and two invertible linear transformations: $S: \FF_2^n \to \FF_2^n$, namely a change of coordinates, and, given a system of fully quadratic equations $\QQ$, a map $\Lambda$ mixing $F$ with $\QQ$. The public key will be
\[
\PP =\begin{cases}
    F\circ S +\Lambda \QQ\\
    \QQ
\end{cases}
\]
We describe a key exchange protocol by using OliVier. Furthermore, we motivate some of the designing choices: the possibility of recovering the oil space led us to a particular design given by the introduction of the map $\Lambda$ (see \ref{sec:lambda}) that allows us to derive the fully quadratic equations from a seed and a pseudo-random function. Besides the analysis of some known attacks, such as Support Minors Modeling, XL/Block Wiedemann, and F4/F5, we also studied some other probabilistic attacks linked to some parameters' choices. Moreover, looking for a security level competing with those proposed by NIST, we propose some specific numerical parameters in view of practical applications. To achieve such security level, we studied the degree of regularity and first fall degree of the system for the proposed parameters. Finally, we present some technique to speed up the communication process, by using particular operation properties on $\FF_2$, in sight of a practical implementation and a complexity analysis.



\section{Preliminaries}
In this section, we recap fundamental notions and definitions that are useful in the sequel. 

 We recall that for a finite field $\FF_q$, we consider the quotient ring $S = \FF_q[x_1,...,x_n]/(x_1^q - x_1,...,x_n^q - x_n)$, where $x_i^q - x_i$ are called \textit{field equations}, for $i = 1,...,n$.
In order to work on a graded ring, we consider the ring $R = S^h = \FF_q[x_1,...,x_n]/(x_1^q,...,x_n^q) $ because in this case the ideal $(x_1^q,...,x_n^q)$ is homogeneous (as done in \cite{DG}, \cite{CG}, \cite{DS}). From now on, we refer to $S^h$ as $S_q$.
We recall the definition of the first fall degree
\begin{Definition}
Given a homogeneous ideal $\MF \subset S_q$, the \emph{first fall degree} $d_{\mathrm{fall}}$ is defined as 
\[
d_{\mathrm{fall}} = \min\{d \in \NN : Syz(\MF)_d/T_d(\MF) \neq 0\},
\]
where $Syz(\MF)_d$ is the vector space of homogeneous syzygies of degree $d$ and $T(\MF)_d = T(\MF) \cap Syz(\MF)_d$, where $T(\MF)$ is the submodule of $Syz(\MF)$ of the trivial syzygies of $\MF$.
\end{Definition}
In this paper, we consider MQ-systems on the finite field $\FF_2$.

Our recurring object will be a quadratic polynomial on $\FF_2[x_1,\ldots, x_n]$, say 
\[
p(x_{1},\ldots, x_{n})=\sum\limits_{1\leq  i \leq j \leq n} \alpha_{ij} x_{i}x_{j} +\sum\limits_{i=1}^n \beta_{i} x_i + \gamma.
\]
Without any particular assumption on the $\alpha_{ij}$, we will call such a polynomial a \emph{fully quadratic} polynomial.
    Moreover, let $v, n \in \NN$ be positive integers with $v< n$. A quadratic polynomial of the form
\begin{equation}\label{eq:OV}
    f(x_{1},\ldots, x_{n})=\sum\limits_{i=1}^n \ \ \sum\limits_{j=1}^v \alpha_{i,j} x_i x_j +\sum\limits_{i=1}^n \beta_{i} x_i + \gamma,
\end{equation}
namely with no monomial $x_{i}x_{j}$ for $i,j \in \{v+1,\ldots,n\}$, is called an \emph{Oil and Vinegar} (OV) polynomial. The variables $x_1, \ldots, x_v$ are called \emph{vinegar} variables, while the variables $x_{v+1},\ldots, x_{n}$ are called \emph{oil} variables.

We denote by $\mathbf{x}$ a vector $(x_1,\ldots, x_n)$ in $\FF_2^n$.
Given a quadratic polynomial $q(\mathbf{x})$, one can consider its \textit{polar form}
\begin{equation}\label{eq:polar}
    q'(\mathbf{x},\mathbf{y})=q(\mathbf{x}+\mathbf{y})+q(\mathbf{x})+q(\mathbf{y}).
\end{equation}

Such a map is known to be symmetric and bilinear (see \cite[Theorem 1]{Be1}).

\subsection{Hilbert series and other algebraic invariants}\label{sec:HS&Alg}


Let us consider a generic system of $m$ quadratic equations in $n$ variables, $f_{1}(x_{1},\ldots, x_{n})$, $\ldots,$ $ f_{m}(x_{1},\ldots, x_{n})$, where for $1 \leq h \leq m$ we have 
\[
f_{h}(x_{1},\ldots, x_{n})=\sum\limits_{1\leq  i \leq j \leq n} \alpha^h_{ij} x_{i}x_{j} +\sum\limits_{i=1}^n \beta^h_{i} x_i + \gamma^h.
\]
For any degree $d \in \NN$, we consider the multiplication of any of the $m$ equations by all the monomials in the variables $x_{1},\ldots, x_{n}$ of degree up to $d-2$. The \emph{Macaulay Matrix} at degree $d$, $M_d$, is the matrix whose columns are labelled by the monomials up to degree $d$ and the rows are labelled as the equations discussed above, such that any entry corresponds to the coefficient of the monomial, on the column, in the corresponding equation, on the row. The least degree $d$ such that, after Gaussian elimination, the rows of $M_d$ form a Gr\"obner basis is called \emph{solving degree}. There are also other notions of degrees that allow to obtain relations on the equations of the system, for example the \emph{degree of regularity}, that is defined in \cite{Ba}. To introduce such an invariant, we need some algebraic setting. 

Given $I\subseteq R=\KK[x_1,\ldots, x_n]$ homogeneous ideal, the \emph{Hilbert function} $\H_{R/I} : \mathbb{N} \rightarrow \mathbb{N}$ is defined by 
\[
\H_{R/I} (d) := \dim_\KK (R/I)_d
\]
where $(R/I)_d$ is the $d$-degree component of the gradation of $R/I$, while the \emph{Hilbert-Poincar\'e series} of $R/I$ is
\[
\HS_{R/I} (t) := \sum_{d \in \NN} \H_{R/I}(d) t^d. 
\]
 By the Hilbert-Serre theorem, the Hilbert-Poincar\'e series of $R/I$ is a rational function. In particular, by reducing this rational function we get
\[
\HS_{R/I}(t) = \frac{h(t)}{(1-t)^d}. 
\]
that is called  \emph{reduced Hilbert series}. 
It is well-known that for large $d$ the Hilbert function $\H_{R/I}(d)$ is a polynomial, called \emph{Hilbert polynomial}, denoted by $\HP_{R/I}(t)$. The least $d$ for which $\H_{R/I}(d)=\HP_{R/I}(d)$, is called \emph{index of regularity}. Observe that, $\H_{R/I}(d)= \dim_\KK R_d - \dim_\KK I_d$. According to this expression, the index of regularity can be seen as the smallest $d$ for which rephrased as $I_d=R_d$.
We define the \emph{degree of regularity} of a system of polynomial equations as the index of regularity of $R/(f_1,\ldots, f_m)$. 


From the characterization of regular sequences (see \cite{Ba}), we can obtain the following on quadratic polynomials.
\begin{Proposition}\label{prop:regSeq}
    Let $f_1,\ldots, f_m$ be a homogeneous regular sequence on $\FF_2$ of quadratic polynomials. The following are equivalent:
    \begin{itemize}
        \item the ideal $(f_1,\ldots, f_m )$ has Krull dimension $n-m$;
        \item the Hilbert series of $R/( f_1,\ldots, f_m )$ is
        \[
         \frac{(1+t)^n}{(1+t^2)^m}
        \]
    \end{itemize}
    
\end{Proposition}
One can use the above characterization to determine whether a sequence of polynomial is ``generic'' (see \cite{Ba}).
\begin{Definition}
A sequence of quadratic homogeneous polynomials $f_1,\ldots, f_m$ on $\FF_2$  is called \emph{cryptographic semiregular} if 
\[
\HS_{R/(f_1,\ldots, f_m)}(t)= \frac{(1+t)^n}{(1+t^2)^m}.
\]
\end{Definition}

\subsection{MinRank Problem and Support Minors Modeling}\label{sec:MinRank}
The MinRank problem is the problem of finding a non-trivial linear combination of $K$ matrices $M_1,\ldots, M_{K}$, of size $m\times n$, having rank smaller than or equal to a \emph{target rank} $r\leq \min \{m,n\}$. Namely, in our case, one has to find coefficients $\lambda_i \in \FF_2$ for $i=1,\ldots,K$, not all zero, such that, given $M= \sum\limits_{i=1}^K \lambda_i M_i$, we have $\rank M \leq r$.

In recent papers, Bardet et al. introduced a new technique, called Support Minors Modeling, to solve the problem in an efficient way (see \cite{Ba1}). The idea of this method is that, if $\rank M \leq r$, then $M$ can be factorized as follows:
\[
M= R \cdot C,
\]
where $R$ is a $m \times r$ matrix and $C$ is a $r \times n$ matrix and they can be full rank. $R$ is an unknown matrix whose columns give a basis for the column space of the matrix $M$. The $j$-th column of $C$ represents the coordinates of the $j$-th column of $R$ with respect to such a basis.
Since the $j$-th row $\mathbf{r}_j$ of $M$ is a combination of the rows of $C$ with respect to the $j$-th row of $R$, and the rank of $C$ is at most $r$, then any $(r+1)$-minor of the $(r+1)\times n$ matrix 
\[
C_{j}=\begin{pmatrix}
\mathbf{r}_j \\
C
\end{pmatrix}
\]
vanishes. Each of these minors arises from a subset of cardinality $r+1$ of the columns of $C_j$. Moreover, expanding the determinant by the row $\mathbf{r}_j$, each term is the product of an element of $\mathbf{r}_{j}$ (hence a linear combination of the coefficients $\lambda_i$) times an $r$-minor of $C$. Since $C$ is an $r\times n$ matrix, the $r$-minors of $C$ are in bijection with the cardinality-$r$ subsets of $\{1,\ldots, n\}$, i.e. given $T \subseteq \{1,\ldots n\}$ with $|T|=r$, we denote by $c_T$, the minor of $C$ containing the columns with indices in $T$. We consider the minors $c_T$ as variables of the system. 
The vanishing of all the $(r+1)$-minors of $C_j$ for $j \in \{1,\ldots, m\}$, gives rise to a bilinear system of $m\binom{n}{r+1}$ equations in the  $\binom{n}{r}+K$  variables $c_T$ and coefficients $\lambda_i$, yielding $ K\binom{n}{r}$ quadratic monomials. Therefore, the system arising from a MinRank instance can be clearly reduced to a linear homogeneous system when
    \begin{equation}\label{eq:minrank}
        m\binom{n}{r+1}\geq K\binom{n}{r} -1.
    \end{equation}
    
    It is worth noting that it may be advantageous to take a \emph{puncturing} of the matrices of the linear combination, that is consider only a fixed subset of $n'$ columns of the original matrices $M_i$. The smallest $n'$ satisfying Equation \ref{eq:minrank} is $n'=2r+2$. According to \cite{Ba1}, in such a case the complexity is
    \begin{equation}\label{eq:MRcompl}
        O \left( K(r+1)\Bigg( \binom{2r+2}{r}K \Bigg)^2\right)= O \left( K^3 (r+1) \binom{2r+2}{r}^2 \right)
    \end{equation}

\section{First fall degree of OV and mixed systems}\label{sec:ffdeg}
In this section, we follow the notation of the paper \cite{CFR} and we discuss the first fall degree of OV systems and mixed systems. In particular, one may observe that a degree fall in the OV system is also a degree fall in a related mixed system, hence the first fall degree of the OV system represents an upper bound for the first fall degree of the mixed system. In view of this, we focus on the first fall degree of OV-systems.

We recall that if $\MF$ is a system of homogeneous quadratic OV-equations with $v$ vinegar variables, then the Hilbert series of $R/\MF$ depends on the Hilbert series of the vinegar ideal $\MV=(x_1,\ldots,x_v)$ modulo $\MF$. In particular, the least $d \in \NN$, such that
\[
\dim \MF_d \geq \dim \MV_d
\]
corresponds to the degree of regularity $d_\reg(\MF)$. In particular, we say that the equations in $\MF_d$ saturate the space $\MV_d$. However, when working in finite fields, we may have saturation of smaller spaces in degrees lower than $d_\reg (\MF)$. We focus on the case $\KK=\FF_2$ with the field equations. 
Even though the argument below holds in general, for our purpose, it makes sense to consider a over/determined system $\MF$ and hence from \cite{CFR} we have $d\leq d_\reg (\MF)-1\leq v$ from now on. 

The space of degree-$d$ monomials can be partitioned as follows 
\[
R_d =\sum\limits_{i=0}^d \MM^i_d,
\]
where $\MM^i_{d}$ is the set of degree-$d$ monomials having exactly $i$ oil variables. We also consider 
\[
\MN^i_{d}=\bigcup\limits_{j=0}^i \MM^j_d.
\]
Since the equations are quadratic, we will multiply by the degree $d-2$ monomials of $R$ to obtain the equations at degree $d$. Moreover, we recall that the monomials of the equations in $\MF$ have either $2$ vinegar variables or $1$ oil and $1$ vinegar variables; so, by multiplying an equation by a monomial in $\MM^i_{d-2}$, we obtain monomials in $\MM^i_{d}$ and $\MM^{i+1}_{d}$.  

Similarly, we observe that multiplying the equations in $\MF$ by monomials in $\MN^i_{d-2}$, we obtain equations in $\MN^{i+1}_{d}$. We highlight that the number of linearly independent equations in each space can be computed similarly as in \cite[Section 2.5]{CFR} through the formula 
\[
\sum\limits_{j=0}^{i} (-1)^{j} \eta^{i-j}_{d-2j-2} \binom{e+j}{j},
\]
where $\eta^{k}_{h}=|\MN^{k}_{h}|$ and $e$ is the number of equations in $\MF$. We set 
\[
\sigma^d_i= \eta^{i+1}_{d}-\sum\limits_{j=0}^{i} (-1)^{j} \eta^{i-j}_{d-2j-2} \binom{e+j}{j}.
\]
As explained above, when $\sigma^d_i <0$ for some $i$, then we have a degree fall at degree $d$, and this is evident in the Hilbert series. In fact, let $\mathcal{G}$ be a generic system with the same number of equations of $\MF$, $\HS_{R/\mathcal{G}}(t)=\sum_{d} h_d t^d$ and $\HS_{\mathcal{F}}(t)=\sum_{i} h'_d t^d $. Then, $\sigma^d_i <0$ for some $i$ and $d$, then $h'_d=h_d - \sigma_i$.

\begin{Example}
    We consider a generic system $\mathcal{G}$ of $25$ quadratic equations in $20$ variables and an OV-system $\MF$ of $25$ quadratic equations in $20$ variables with $9$ vinegar variables. By using \texttt{MAGMA}, we compute $d_\reg (\mathcal{G})=d_{\reg}(\MF)=5$ and we compute $\sigma^d_i$ for $i \in \{0,\ldots, d-2\}$ and $d\in \{3,4\}$. We exhibit the results of computation
\begin{table}[H]
    \centering
    \begin{tabular}{|c|c|c|c|c|}
    \hline
         &0&1&2&3  \\
         \hline
         3&255& 475& 640& \\
         4&150 &$-$20 &90 &420\\
         \hline
    \end{tabular}
    \caption{The $\sigma^d_i$ for $i \in \{0,\ldots, d-2\}$ and $d\in \{3,4\}$ }
    \label{tab:sigdi}
\end{table}
    We have a degree fall at degree $4$, and we spot this fall also in the Hilbert series.
In fact computing the two Hilbert Series 
    \[
    \HS_{R/\mathcal{G}}(t)=1 + 20t + 165t^2 + 640t^3 + 420t^4
    \]
    \[
     \HS_{R/\mathcal{F}}(t)=1 + 20t + 165t^2 + 640t^3 + 440t^4+\ldots.
    \]
    the coefficient at degree $4$ is increased by $-\sigma^4_1=20$.
    
\end{Example}

We now consider a mixed system $\PP=\MF+\QQ$, where $\QQ$ is a fully quadratic system with $u$ equations. The first fall degree estimated on $\MF$ through $\sigma_i$ represents an upper bound for the first fall degree of the system $\PP$, namely
\[
d_{\mathrm{fall}}(\PP)\leq d_{\mathrm{fall}}(\MF).
\]
However, also the OV-monomials of the equations of $\QQ$ affect $\sigma_i$ for $i \geq 1$. That is, they appear in spaces whose monomials have at most $i+1$ oil variables. It worth noting that it is possible to have
\[
d_{\mathrm{fall}}(\PP) < d_{\mathrm{fall}}(\MF),
\]
as shown in the next example. 
\begin{Example}
    Set $n=24$.    We consider a generic system $\mathcal{G}$ of $34$ quadratic equations and a mixed system $\PP=\MF+\QQ$ where $\MF$ is an OV-system of $24$ (quadratic) equations with $6$ vinegar variables and $\QQ$ is composed by $10$ fully quadratic equations. By using \texttt{MAGMA}, we compute $d_\reg (\mathcal{G})=d_{\reg}(\PP)=5$ and we compute $\sigma^d_i$ as above for $i \in \{0,\ldots, d-2\}$ and $d\in \{3,4\}$. We exhibit the results of computation
\begin{table}[H]
    \centering
    \begin{tabular}{|c|c|c|c|c|}
    \hline
         &0&1&2&3  \\
         \hline
         3&146& 632& 1448& \\
         4&15& 18& 1242& 4302\\
         \hline
    \end{tabular}
    \caption{The $\sigma^d_i$ for $i \in \{0,\ldots, d-2\}$ and $d\in \{3,4\}$ }
    \label{tab:sigdi}
\end{table}

    We observe that in $\MF$ there is no degree fall.  Moreover, the $10$ fully quadratic equations can be multiplied by two vinegar variables, obtaining $\binom{6}{2} \cdot 10=150$ degree-4 relations, that are all independent. These relations live in a space that has at most $2$ oil, namely the one related to $\sigma^4_1=18$. Set $\widetilde{\sigma}^4_{1}=18-150=-132$.
    
We compute the Hilbert Series of $\mathcal{G}$ and $\PP$:
    \[
    \HS_{R/\mathcal{G}}(t)=1 + 24t + 242t^2 + 1208t^3 + 1837t^4
    \]
    \[
     \HS_{R/\mathcal{\PP}}(t)=1 + 24t + 242t^2 + 1208t^3 + 1969t^4+\ldots.
    \]
    One may observe that the coefficient at degree $4$ is increased by $-\widetilde{\sigma}^4_{1}=132$.
    
\end{Example}

To have a precise computation of the number $\widetilde{\sigma}^d_{i}$, we use the usual technique of counting by inclusion-exclusion principle, to obtain the following expression
\[
\widetilde{\sigma}_i^d=\sigma_{i}^d + \sum_{j=1}^{i} (-1)^j(A_j+B_j),
\]
where for any $j$
\[
A_j=\begin{cases}
    \eta_{i+1-2j}^{d-2j}\binom{u+j-1}{j} &\mbox{if } i+1-2j\geq 0\\
    0 &\mbox{otherwise};
\end{cases}
\]
that is the number of relations that involve only fully quadratic equations, and 
\[
B_j=
    \sum_{k+h=j,\\ i+1-2h-k\geq 0} \eta_{i+1-2h-k}^{d-2j}\binom{u+h-1}{h}\binom{e+k-1}{k},
\] 
that is the number of relations that involves products of OV and fully quadratic equations.

\section{Description}\label{sec:Desc}
Let $v, n \in \NN$ be positive integers with $v<<n$ and $n \geq 1$, and let $o=n-v > v$. We consider a quadratic map
\[
F: \FF_{2}^n \to \FF_2^{n}
\]
given by quadratic polynomials, i.e.  $F(x_1\ldots,x_{n})$ $=$ $(f_1(x_1,\ldots,x_n),\ldots, f_{n}(x_1,\ldots,x_n))$ where, for $i\in \{1,\ldots, n\}$ $f_{i}(x_1,\ldots, x_n)$ is Oil $\&$ Vinegar type  with oil variables $x_{v+1}$, $\ldots$, $x_{n}$, as in Equation \eqref{eq:OV}.

We take an invertible linear application, $S \in \GL_{\FF_2}(n)$ and, for any $i \in \{1,\ldots, n\}$, we set $g_{i}(x_1,\ldots,x_n)=f_{i}(S(x_1,\ldots,x_n))$. Moreover, for a positive integer $u$, we take $u$ fully quadratic polynomials that we label by $q_{1}(x_1,\ldots,x_n)$, $\ldots$,  $q_{u}(x_1,\ldots,x_n)$, and we set $m=n+u$ .
Let $\Lambda=(\lambda_{ij})$ be a $n \times u$ matrix on $\FF_2$.

The public key of OliVier is $\PP(x_1\ldots,x_{n})$ $=$ $(p_1(x_1,\ldots,x_n),\ldots, p_{m}(x_1,\ldots,x_n))$, where, for $i \in \{1,\ldots, n\}$, we have 
\[
p_{i}(x_1,\ldots,x_n)=g_{i}(x_1,\ldots,x_n) + \sum_{j=1}^{u} \lambda_{ij} q_{j}(x_1,\ldots,x_n).
\]
and, for $i \in \{n+1,\ldots, m\}$,
\[
p_{i}(x_1,\ldots,x_n)=q_{i-n}(x_1,\ldots,x_n).
\]

Suppose that the matrix $\Lambda$ is invertible. Otherwise, one can easily remove the fully quadratic in $\approx 2^{(u-1)}$ attempts part from $p_1,\ldots, p_{n}$, finding a OV-type equation. It follows that $u\geq n$. In Section \ref{sec:probatt}, we find a probabilistic attack if $u\approx n$.

\subsection{Communication}\label{sec:comm}
Bob publishes $\PP$ and keeps secret $F$, $S$ and $\Lambda$. Alice takes $t$ vectors, $\mathbf{a}_1,\ldots \mathbf{a}_t \in \FF_2^n$, and, for any $i\in \{1,\ldots, t\}$, computes $\mathbf{b}_i=\PP(\mathbf{a}_i)$, and sends the $\mathbf{b_i}$s to Bob. Bob performs the following \emph{guessing algorithm}.

\begin{Algorithm}\label{alg}
\begin{enumerate}
    \item For any $i\in \{1,\ldots, t\}$, given
    \[
    \mathbf{b}_{i}=\begin{pmatrix}
        b_i^{(1)}\\
        \vdots \\
        b_i^{(n)}\\
        b_i^{(n+1)}\\
        \vdots \\
        b_i^{(m)}\\
    \end{pmatrix},
    \]
    Bob computes
    \[
    \mathbf{w}_{i}=\begin{pmatrix}
        w_i^{(1)}\\
        \vdots \\
        w_i^{(n)}
    \end{pmatrix}
    \]
    where
    \[
    w^{(j)}_{i}=b_{i}^{(j)}+ \sum_{k=n+1}^{m} \lambda_{jk} b_i^{(k)},
    \]
    for $j$ in $\{1,\ldots, n\}$. Let $W$ be the $n \times t$ matrix having columns $\mathbf{w}_{1}, \ldots \mathbf{w}_t$. Let us denote by $W_i$ for $i\in \{1,\ldots, n\}$ the rows of $W$.
    \item Given the central polynomials
    \[
    f_{j}(y_1,\ldots, y_{n})
    \]
    for $j$ in $\{1,\ldots, n\}$, Bob performs a \emph{guessing} of the variables $y_{1},\ldots, y_{v}$ that are the vinegar ones, obtaining linear polynomials in the oil variables $y_{v+1},\ldots, y_{n}$. Namely, given $ \ol{\mathbf{y}} \in \FF^v_{2}$, the polynomial
    \[
    \ol{f}_j(y_{v+1},\ldots, y_{n})=f_j(\ol{\mathbf{y}},y_{v+1},\ldots, y_{n})
    \]
   is linear in the variables $y_{v+1},\ldots, y_{n}$. Let $M_{\ol{\mathbf{y}}}$ be the coefficient matrix of these polynomials. Then, $M_{\ol{\mathbf{y}}}$ is a $n \times (n-v)$ matrix, and let us denote its rows by $M_i$ for $i\in \{1,\ldots, n\}$.
   Bob considers the following matricial equation
   \begin{equation}\label{eq:MyW}\tag{$\ast$}
       M_{\ol{\mathbf{y}}}\begin{pmatrix}
        y_{v+1}\\
        \vdots \\
        y_{n}
    \end{pmatrix}=W
 \end{equation}

   The rank of $M_{\ol{\mathbf{y}}}$ is (at most) $n-v$, that is there are (at most) $n-v$ independent rows, say $M_1,\ldots M_{n-v}$. Hence, there exist coefficients $d_{ij}$ for $i \in \{n-v+1, \ldots, n\}$ and $j \in \{1,\ldots, n-v\}$ such that  
   \[
   M_{i}= \sum_{j=1}^{n-v} d_{ij}M_j.
   \]
   If Equation \eqref{eq:MyW} is compatible, then also the rows of $W$ satisfy the above relations, say 
   \[
   W_{i}= \sum_{j=1}^{n-v} d_{ij}W_j.
   \]
   With this check, we discard the spurious solutions.
   Therefore, given $\widetilde{W}$ as the matrix whose rows are 
   \[
   \widetilde{W}_i=\begin{cases}
      W_i &\mbox{if } i \in \{1,\ldots, n-v\}; \\
     \sum\limits_{j=1}^{n-v} d_{ij}W_j &\mbox{if } i \in \{n-v+1,\ldots, n\}
   \end{cases},
   \]

   then, the guessing $\ol{\mathbf{y}}$ is a good candidate for finding a preimage, if there exists a  $k \in \{1,\ldots, t\}$ for which the $k$-th column 
   \[
    \widetilde{\mathbf{w}}_{k}=\begin{pmatrix}
        \widetilde{w}_k^{(1)}\\
        \vdots \\
        \widetilde{w}_k^{(n)}
    \end{pmatrix}
    \] 
   of $\widetilde{W}$ is such that $\widetilde{w}_k^{(j)}=0$ for any $j \in \{n-v+1,\ldots, n\}$. If this is not the case, Bob returns to Step (2) and chooses another $\ol{\mathbf{y}}$.
    \item Bob solves the system:
    \[
    M_{\ol{\mathbf{y}}}\begin{pmatrix}
        y_{v+1}\\
        \vdots \\
        y_{n}
    \end{pmatrix}=\textbf{w}_k,
    \]
    finding a solution $\widetilde{\mathbf{y}} \in \FF_2^{n-v}$. Given $\mathbf{y}=(\ol{\mathbf{y}} \ | \ \widetilde{\mathbf{y}}) \in \FF_2^{n}$ and $\mathbf{x}=S^{-1}\mathbf{y}$, Bob verifies that $\PP(\mathbf{x})=\mathbf{b}_k$ (this check can be skipped due to Section \ref{sec:opt}.(2-3)) and outputs $\mathbf{x}$, otherwise returns to Step (2).
\item Bob sends $H(\mathbf{x})$ back to Alice  for some hash function $H$.
\item If Alice finds $H(\mathbf{x})$ among the hashes of $\mathbf{a}_1,\ldots, \mathbf{a}_t$ then they agree on some $\mathbf{a}_k=\mathbf{x}$.
\end{enumerate}    
\end{Algorithm}

Roughly speaking, Bob performs a guessing on the vinegar variables, solves a linear system recovering some candidates of oil variables associated to the guessed vinegar ones. By running the algorithm on all the possible vectors of $\FF_2^v$, the algorithm terminates.

\section{Designing  motivations}
In this section, we present some observations that arose during the designing of the system and that influenced the choice of parameters. In particular, we will explain how to attack the system by recovering the oil space in two cases: by dropping the map $S$, and if the number of fully quadratic equations $u$ is near to $n$.
 
\subsection{Polar form of Oil and Vinegar equation}\label{sec:Polar}
In this subsection, we analyze the properties of the polar form of the central map and after the change of coordinates.
Let us consider the quadratic map $\QQ=F\circ S=(g_1,\ldots, g_{n})$, with $F=(f_1,\ldots, f_{n})$. From Equation \eqref{eq:OV}, for $i\in \{1, \ldots, n \}$, the polar form $f'_{i}(\mathbf{x},\mathbf{y})$ is a homogeneous degree-2 polynomial with monomials $x_iy_j$ with $\{i,j\} \not\subseteq \{v+1,\ldots, n\}$. From this property, we get that $f'_i(\mathbf{x},\mathbf{y})$ vanishes if both $\mathbf{x}$ and $\mathbf{y}$ lie on the space:
\[
\MO'=\{(w_{1},\ldots, w_{n}) \in \FF_{2}^n \ | \ w_{k}=0 \mbox{ for } k \in \{1,\ldots,v\}\}
\]
and, by setting $\MO=S^{-1} \MO'$, we get that $g'_i(\mathbf{x},\mathbf{y})=0$ for any $\mathbf{x},\mathbf{y} \in \MO$. \\

From a matricial point of view, one can write 
\[
f_i(\mathbf{x})=\mathbf{x}^t N_i \mathbf{x},\ \  g_i(\mathbf{x})=\mathbf{x}^t M_i \mathbf{x},
\]
and 
\[
f'_i(\mathbf{x},\mathbf{y})=\mathbf{x}^t N'_i \mathbf{y}, \ \ g'_i(\mathbf{x},\mathbf{y})=\mathbf{x}^t M'_i \mathbf{y},
\]
where $M'_i=M_i^t+ M_i$ and $N_i'=N_i^t+ N_i$.
\begin{Remark}\label{rmk:Ni'}
We observe that, for any $i\in \{1,\ldots, n\}$, the matrix $N_i'=(\nu_{jk})$ is such that  $\nu_{jk}=0$ for $\{j,k\} \subseteq \{v+1,\ldots, n \}$, hence
    \[
    \begin{pmatrix}
        \nu_{11} &\nu_{12} &\ldots &\nu_{1v} &\nu_{1v+1} &\ldots &\nu_{1n}\\
        \vdots & \vdots &\ldots & \vdots & \vdots & \vdots & \vdots\\
        \nu_{v1} &\nu_{v2} &\ldots &\nu_{vv} &\nu_{vv+1} &\ldots &\nu_{vn}\\
        \nu_{v+11} &\nu_{v+12} &\ldots &\nu_{v+1v} &0 &\ldots &0\\
        \vdots & \vdots &\ldots & \vdots & \vdots & \ddots & \vdots\\
        \nu_{n1} &\nu_{n2} &\ldots &\nu_{nv} &0 &\ldots &0\\
    \end{pmatrix}
    \]
\end{Remark}
\begin{Lemma}\label{lem:rkNi}
The rank of the matrix $N_i'$ is less than or equal to $2v$.
\end{Lemma}
\begin{proof}
Since $v<<n$, then one can find a minor $v\times v$ on the first $v$ rows and in the last $n-v$ columns of $N'_i$ having non-zero determinant, as well as a minor on the first $v$ columns and in the last $n-v$ rows, by symmetry. Hence the rank is less than or equal to $2v$.
\end{proof}

\begin{Lemma}
    For any $i \in \{1, \ldots, n \}$, we have 
    \[
    M'_i= S^t N_i' S
    \]
\end{Lemma}
\begin{proof}
    From $F ( S (\mathbf{x} ))= \PP(\mathbf{x})$, we get that, for any $i \in \{1, \ldots , m\}$, the matrices $S^t N_i S$ and  $M_i$ yield the same quadratic polynomial. Therefore, $A= S^t N_i S - M_i$ is the matrix of the 0-polynomial and hence is a symmetric matrix. Furthermore, $S^t N'_i S - M'_i=A^t +A= 0$.

\end{proof}

\begin{Lemma}\label{lem:kerNi}
Let $i \in \{1, \ldots, n\}$, be such that $\rank{N'_i}=2v$. Then, the following hold:
\begin{enumerate}
    \item $\ker(N'_i)\subset \MO'$. 
    \item $\ker(M_i')= S^{-1} \ker(N_i')$ and hence $\dim (\ker(M'_i))=n-2v$. 
\end{enumerate}
\end{Lemma}
\begin{proof}
   
    (1).  Let $N'_i= (\nu_{jk})$. We take $\mathbf{w}\in \ker(N'_i)\cap \MO'$, that is $\mathbf{w}=(0,\ldots, 0, w_{v+1}\ldots, w_{n})$ and $N'_i \mathbf{w}= \mathbf{0}$. 
    From the latter and Remark \ref{rmk:Ni'}, we get that
    \[
    \begin{cases}
        \sum\limits_{k=v+1}^{n} \nu_{1k} w_{k}& = 0 \\
        &\vdots\\ 
        \sum\limits_{k=v+1}^{n} \nu_{vk} w_{k} &= 0
    \end{cases}
    \]
    is a system of $v$ equations in the $n-v$ indeterminates $w_{v+1}\ldots, w_{n}$. Since $\rank{N'_i}=2v$, then the matrix associated to the above system has rank $v$, that is the solution space has dimension $n-2v$. By hypothesis, $\rank{N'_i}=2v $, so we have that $ \dim \ker(N_i') = n-2v$, and it follows that $\ker(N'_i)\subset \MO'$.\\
    (2). Follows from (1) and the fact that $S$ is an invertible matrix.
\end{proof}

We observe that since $v << n-v$, the condition $\rank N_i'< 2v$ happens with negligible probability.
Therefore, from Lemma \ref{lem:kerNi}, it follows that an attacker can recover an $E=\bigcup\limits_{i=1}^{n} \ker(M'_i) \subset S^{-1}\MO'$ with $|E|\approx n2^{(n-2v)}$ that is \emph{fatal} information leakage. Moreover, considering $2^v$ combinations of the $M'_i$ we may obtain a subset of cardinality $2^v \cdot 2^{(n-2v)}$, and hence recover the full oil space $\MO$.\\ 

Note that the map $\QQ$ discussed in this section has the form of the one described in the paper \cite{Be1}, with the difference that in their case $n-v < v$ and hence $n-2v < 0$. In this section, we analysed that the MQ-system obtained by combining the central map with only a change of coordinates, i.e. the right composition with $S$, gives rise to matrices of small rank whose kernels have a large intersection with the oil space. To avoid this kind of leakage of information, in our description we have added the fully quadratic equations together with the map $\Lambda$ to combine $OV$-type equations with fully quadratic ones.

\subsection{The role of $\Lambda$}\label{sec:lambda}

In this subsection, we show that our public key $\PP$ obtained by multiplying $\Lambda ' \in \FF_2^{(m \times m)}$, where

$$\Lambda ' = \begin{pmatrix}
 I & \Lambda\\
 0 &  I
\end{pmatrix} $$

by $ \begin{pmatrix}
    F\\
    \QQ
\end{pmatrix}$ is equivalent to apply a left (invertible) transformation $T$ on 
\[ 
\begin{pmatrix}
    F\\
    \QQ
\end{pmatrix}.
\]
Observe that if
 $T =  \begin{pmatrix}
 T_1 & T_2 \\
 T_3 & T_4
\end{pmatrix}$
is an invertible transformation, each block $T_i$, for $i =1, 2, 3, 4$, is an invertible submatrix. The product $$T \cdot \begin{pmatrix}
    F\\
    \QQ
\end{pmatrix}
  = \begin{pmatrix}
    T_1 F + T_2 \QQ \\
    T_3 F + T_4 \QQ
\end{pmatrix}
$$
is still a system made by $n$ OV equations combined with fully quadratic equations and $u$ fully quadratic ones. 
Let
$\begin{cases}
    F' = T_1 F + T_2 \QQ \\
    \QQ' = T_3 F + T_4 \QQ
\end{cases}$,
we have that $\QQ = T_4^{-1} (\QQ' - T_3 F)$, and so $F' = (T_1 - T_2 T_4^{-1} T_3) F + T_2 T_4^{-1} \QQ$, that is, $F'$ is composed by $n$ OV equations since $(T_1 - T_2 T_4^{-1} T_3) F$ is a linear combination of OV equations mixed with $u$ fully quadratic $T_2 T_4^{-1} \QQ$, while the block $\QQ' = T_3 F + T_4 \QQ$ is composed by $u$ fully quadratic equations.
Moreover, we can assume that the fully quadratic equations are generated by an initial seed of the system; this is the reason why we use the matrix $\Lambda '$ instead of the complete transformation $T$. In this way, we reduce memory space because the block of the fully quadratic equations is stored in a compact form.

\subsection{Probabilistic attack: reconstruction of the Oil space}\label{sec:probatt}
We highlight that for some choices of the parameter $u$, one could attack our public key $\PP$ in a probabilistic way. 
Since we operate on the field $\FF_2$, for any homogeneous quadratic polynomial $p({\bf{x}}) \in \FF_2[x_1,\ldots,x_n]$ and for any ${\bf a} \in \FF_2^n$, Equation \eqref{eq:polar} yields 
\[
p'({\bf a},{\bf a})=2p({\bf a})+p({2\bf a})=0.
\]
We consider a fully quadratic system $\QQ$ of $n$ equations in $n$ variables. 
Our aim is to find pairs $({\bf a},{\bf b})\in \FF_2^n$ such that 
\[
\QQ'({\bf a},{\bf b})={\bf 0}.
\]
Such pairs will be fundamental to run the probabilistic attack. One may wonder how many of such pairs exist. Given a subset $A \subset \FF_2^n$ of vectors, for any ${\bf a}\in A$ we solve the linear system of equations
\[
\QQ'({\bf a},{\bf x})={\bf 0}
\]
for ${\bf x} \notin \{{\bf 0},{\bf a}\}$. In particular, the required vector ${\bf b}$ exists whenever the rank of the associated matrix of the above system is strictly less than $n-1$. This is the case for about $40\%$ of the vectors ${\bf a}$ of $A$.
One may observe that increasing the number of equations by $1$, the probability of finding such a pair halves. In particular, given two fully quadratic systems $\QQ_1, \QQ_2$ of $n$ equations each, it is unlikely to find a pair of vectors $({\bf a},{\bf b})$ such that 
\[\QQ_1'({\bf a},{\bf b})={\bf 0} \mbox{ and } \QQ_2'({\bf a},{\bf b})={\bf 0}. \]

We recall that the Oil \& Vinegar map $\MF=F\circ S$ vanishes on all the oil space $\MO$. In particular, for any ${\bf a},{\bf b} \in \MO$, we have 
\[
\MF'({\bf a},{\bf b})={\bf 0}.
\]
We also recall that our public key $\PP$ has the following structure
\[
\PP= \begin{cases}\MF+ \Lambda \QQ \\ \QQ \end{cases},
\]
where $\MF=F \circ S$. We look for the pairs $({\bf a},{\bf b})$ such that 
\[
\PP'({\bf a},{\bf b})={\bf 0}.
\]
Our computation shows that even though $\PP$ is made by $\approx$ $ 2n$ equations, then such pairs exist. In particular, they satisfy 
\[\MF'({\bf a},{\bf b})={\bf 0} \mbox{ and } \QQ'({\bf a},{\bf b})={\bf 0}, \]
that, as discussed above, is unlikely whenever $\MF({\bf a}),\MF({\bf b})\neq {\bf 0}$. Therefore, solving $\PP'({\bf x},{\bf y})={\bf 0}$, we find pairs $({\bf a},{\bf b})$ such that \[
\QQ'({\bf a},{\bf b})={\bf 0} \mbox{ and } \MF({\bf a})=\MF({\bf b})=\MF({\bf a}+{\bf b})={\bf 0},
\]
that is ${\bf a},{\bf b} \in \MO$ with high probability.
However, once we have found one of these pairs $({\bf a},{\bf b})$, there is no direct way to check whether $F$ vanishes on ${\bf a}$ or ${\bf b}$. Hence, we generate enough pairs, $({\bf a}_1,{\bf b}_1),\ldots, ({\bf a}_t,{\bf b}_t)$ with $t\geq \frac{n-v}{2}$ and we check whether
\[
\dim \mathrm{Span}<{\bf a}_1,{\bf b}_1,\ldots,{\bf a}_t,{\bf b}_t> =n-v.
\]
If this is the case, we firstly find a basis ${\bf w}_1,\ldots, {\bf w}_{n-v}$ for the above span, we secondly generate $n+\sigma$ linear combinations of ${\bf w}_1,\ldots, {\bf w}_{n-v}$, say ${\bf v}_1,\ldots, {\bf v}_{n+\sigma}$,  with $\sigma \approx 5$, and, letting 
\[
B= \begin{pmatrix}
    \PP({\bf v}_1) \\
     \PP({\bf v}_2) \\
     \cdots \\
      \PP({\bf v}_{n+\sigma})
\end{pmatrix},\]
we finally check whether
\[
\rank B = n.
\]
We highlight that $B$ is a $ (n+\sigma)\times 2n$ matrix, and for a vector ${\bf v} \in \MO$ one has 
\[
\PP({\bf v})= \begin{cases} \Lambda \QQ({\bf v}) \\ \QQ({\bf v}) \end{cases},
\]
hence, if $\rank B = n$ then ${\bf v}_i \in \MO$ for $i \in 1,\ldots, n+\sigma$ and $\MO = <{\bf w}_1,\ldots, {\bf w}_{n-v}> $.

We resume the attack in the following 
\begin{Algorithm}
  \begin{enumerate}
    
      \item Generate a subset $A \subset \FF_2^n$; 
      \item For any ${\bf a}\in A$ solve the linear system of equations
\[
\PP'({\bf a},{\bf x})={\bf 0}
\]
for ${\bf x} \notin \{{\bf 0},{\bf a}\}$.
\item Let $({\bf a}_1,{\bf b}_1),\ldots, ({\bf a}_t,{\bf b}_t)$ be the pairs found in step (2). If $2t < n-v$, then return to step (1). Otherwise 
check
\[
\dim \mathrm{Span}<{\bf a}_1,{\bf b}_1,\ldots,{\bf a}_t,{\bf b}_t> =n-v.
\]
If this is the case, take a basis ${\bf w}_1,\ldots, {\bf w}_{n-v}$. Otherwise, return to step (1).
\item Generate $n+\sigma$ linear combinations  ${\bf v}_1,\ldots, {\bf v}_{n+\sigma}$ of  the vectors ${\bf w}_1,\ldots, {\bf w}_{n-v}$, with $\sigma \approx 5$, and check 
\[
 \rank \begin{pmatrix}
    \PP({\bf v}_1) \\
     \PP({\bf v}_2) \\
     \cdots \\
      \PP({\bf v}_{n+\sigma}) 
\end{pmatrix} =n. 
\]
If this is the case, output $\MO=<{\bf w}_1,\ldots, {\bf w}_{n-v}>$, otherwise, return to step (1).
  \end{enumerate}  
\end{Algorithm}

\section{Known Attacks}\label{sec:attacks}
In this section, we analyze which are the most known attacks that an attacker Eve can perform on our cryptosystem. In our opinion, the ones that are significant are three and we will denote them  by $E_1, E_2,$  and $E_3$. The first one aims to remove the fully quadratic equations, to obtain an OV-type polynomial and proceed as discussed in Section \ref{sec:Polar}, while the other ones regard our system as a general one. The three attacks assume that Eve knows a ciphertext $\mathbf{b}$ and wants to find the associated plaintext $\mathbf{a}$.

\begin{itemize}

    \item[$(E_1)$]  \textbf{Solving a MinRank instance}: 
    as pointed out in Section \ref{sec:Polar}, the associated matrices to the polar form of an OV-type polynomial with $v$ vinegar variables have rank less than or equal to $2v$. 
    Some linear combinations of one of the $p'_{i}$ for $i \in \{1,\ldots, n\}$ and  $p'_{n+1},\ldots, p'_{m}$ have an associated matrix of rank less than or equal to $2v$, giving rise to a MinRank instance with $K=u+1$, target rank $r=2v$, and matrices of size $n\times n$. Solving this instance, Eve finds the OV-type equation $g_{i}$.
   
    According to Equation \eqref{eq:MRcompl} with $r=2v$, by puncturing the matrices to $n'=4v+2$,  the complexity is
    \[
     O \left( (u+1)^3 (2v+1) \binom{4v+2}{2v}^2 \right)
    \]
    \end{itemize}
     Eve may ignore the properties of the system and solve it by Gr\"obner Bases techniques. 
    \begin{itemize}
   
    \item [$(E2)$] \textbf{XL/Block Wiedemann}:
    Eve can perform the Block Wiedemann XL algorithm, described in \cite{CCNY, Be2}, that is based on the linearization of the system and, after obtained a large but very sparse system of linear equations, solves the system by taking advantage of the sparsity  by using the Block Wiedemann technique. The estimated complexity is 
    \[
     O \left( 3\binom{n}{D_{XL}}^2\binom{n}{2} \right)
    \]
    where $D_{XL}$ is the degree of regularity of 
    \[
\frac{1}{(1-t)}\HS_{\PP}(t)
    \].
      \item[$(E_3)$] \textbf{F4/F5 algorithms}:
      finally, the system can be solved by Gr\"obner basis computation, with the F4/F5 algorithms that retrieve it by applying Gaussian elimination to the Macaulay matrix. Their complexity is 
    \[
     O \left( \binom{n}{d_{\reg}(\PP)}^\omega \right)
    \]
    where $d$ is the degree of regularity of the system (see \cite{YCY}).
\end{itemize}

\section{Proposed Parameters}
In this section, we provide three different sets of parameters for OliVier with the aim of challenging the NIST security levels 1,3 and 5. They require a bit-security of $128$, $192$, $256$ bits respectively that correspond to bit-complexities of  $143$, $205$, $272$ bits. Moreover, Algorithm \ref{alg} complexity is $O(n\cdot 2^v)$, and we can keep the number of vinegar variables relatively small with respect to the size of $n$. First of all, from Section \ref{sec:probatt}, the number $u$ of fully quadratic must be significantly greater than $n$, and for our choices we set $u=2n$, that is the public key $\PP$ has $m=n+u=3n$ equations. We choose $n=320,640,1280$. We use \texttt{MAGMA} to find the values that attain the aforementioned complexity bounds.
By plugging these values of $n$ in the complexities of attacks (E1), (E2), (E3), we obtain that the operating degrees to have the desired complexities are $d=12,14,17$. Hence, for the three choices of $n$, we want to keep the first fall degree of the system $\PP$ equal to $d=12,14,17$. By computing  $\widetilde{\sigma}_i^d$ described in Section \ref{sec:ffdeg}, we find that the corresponding vinegar choices are $v=24,29,36$.

We resume the parameter choices in the Table \ref{tab}, reporting the complexities of the attacks described in section \ref{sec:attacks}. In the table, $v$ is the number of the vinegar variables, MR stands for Min Rank attack, $d_{\mathrm{fall}}$ denotes the first fall degree, XL indicates the XL algorithm and F4 the Faugeré algorithm complexities.

\begin{table}[H]
\begin{tabular}{|c|c|c|c|c|c|c|c|}
\hline
& Target & $n$ &$v$ &MR &$d_{\mathrm{fall}}$ &XL &F4\\  
\hline
SL1&128 (143) &320  &24 &222 &12 &158 &167\\
\hline
SL2&192 (205) &640  &29 &265 &14 &207 &222\\
\hline
SL2&256 (272) &1280  &36 &324 &17 &275 &300\\
\hline
\end{tabular}
\caption{The complexity estimates for the proposed parameters}\label{tab}
\end{table}

\section{Optimization the communication process}\label{sec:opt}
We present several techniques that can be used to speed up the communication process described in Section \ref{sec:comm} and arise from the proposed parameters. 
\begin{enumerate}
    \item The number $t$ can be chosen to increase the probability of finding a guess by using birthday paradox. In particular, Alice chooses $t$ as multiple of $64$ for system requirements. 
    \item To perform the guessing, for any choice of the vinegar variables, Bob precomputes the coefficients $d_{ij}$ in Algorithm \ref{alg}. The amount he needs to precompute is $2^{v-\log_2 t}$. The linear combinations of the $W$ related to the coefficients $d_{ij}$ can be computed by the usage of the Gray code \cite{B} and a technique derived by the  \textit{Method of Four Russians} \cite{ADKF}. The idea is to choose an optimal size $s$ (possibly dividing $n-v$) and compute all the possible $2^s$ binary vectors and enumerating them in an efficient way, namely in any vector by exactly one bit in one position from each of its neighbors.
    The chunks $s$-by-$s$ of $W$ are $\{W_{(h-1)s+1},\ldots, W_{(h-1)s+s}\}$ with $h \in \{1,..,\frac{n-v}{s}\}$.
    For any $\mathbf{e}=(e_1,\ldots, e_s) \in \FF_2^s$ and any $h \in \{1,..,\frac{n-v}{s}\}$ let $C^{(h)}_{\bf{e}}$ be the linear combination  $\sum_{j=1}^s e_jW_{(h-1)s+j}$.
    Therefore, computing $\widetilde{\mathbf{w}}_{k}$ is equivalent to perform $\frac{n-v}{s}$ vector additions of the $C^{(h)}_{\mathbf{e}_h}$ for some $\mathbf{e}_h$ depending on $d_{ij}$. 
    For what concerns the storage, this results in storing $2^s \cdot \frac{n-v}{s}$ linear combinations.
    
   \item To find the $\widetilde{\mathbf{w}}_{k}$ of Algorithm $\ref{alg}$, Bob computes such linear combinations only for few values $\theta\approx 8$ of and verifies whether the $\mathrm{OR}$ of these rows is $\mathbf{1}$. If this is the case, then the guessing is wrong. This happens because, for a wrong guessing there is $\approx \frac{1}{2}$ to obtain $0$ or $1$ in the row $\widetilde{W}_{n-v+1}$, that decreases to $\approx \frac{1}{4}$ for the row $\widetilde{W}_{n-v+2}$ and so on. For $\theta =8$,  a wrong solution verifies the equations  $\widetilde{W}_{n-v+1}, \ldots,  \widetilde{W}_{n-v+\theta}$, with probability $\approx 1/2^8=0.015$. 
   \item From the last item, we get that the number of operations needed to verify each vinegar guess is
   \[
   \frac{n-v}{s}\cdot \theta +\theta,
   \]
   resulting in a maximum of
   \[
   2^{v-\log_2 t} \left(  \frac{n-v}{s}\cdot \theta +\theta \right)
   \]
   operations.  
   
\end{enumerate}

\subsection*{Acknowledgments}
 This paper was made possible through the financial assistance provided by the Project ECS 0000024 “Ecosistema dell’innovazione - Rome Technopole” financed by EU in NextGenerationEU plan through MUR Decree n. 1051 23.06.2022 PNRR Missione 4 Componente 2 Investimento 1.5 - CUP H33C22000420001.

\end{document}